\documentclass[12pt]{article}

\usepackage[affil-it]{authblk}
      \usepackage{latexsym}
         \usepackage[reqno, namelimits, sumlimits]{amsmath}
         \usepackage{amssymb, amsfonts}
         \usepackage{amsthm}

 \newtheorem{theorem}{Theorem}[section]
 
 \newtheorem{definition}[theorem]{Definition}
 
 \newtheorem{lemma}[theorem]{Lemma}
 
 \newtheorem{remark}[theorem]{Remark}
 
 \newtheorem{cor}[theorem]{Corollary}
 \newtheorem{pro}[theorem]{Proposition}

\title{  On global solutions to the Navier-Stokes system with large $L^{3,\infty}$ initial data}

\author{T Barker, G Seregin
  \thanks{Email addresses: \texttt{tobias.barker@seh.ox.ac.uk, seregin@maths.ox.ac.uk}; }}
\affil{OxPDE, Mathematical Institute, University of Oxford, Oxford,UK}
\date{ \today}

\begin{document}
\maketitle
\begin{abstract}
This paper addresses a question concerning the behaviour of a sequence of global solutions to the Navier-Stokes equations, with the corresponding sequence of smooth initial data  being bounded in the (non-energy class) weak Lebesgue space $L^{3,\infty}$. It is closely related to the question of what would be a reasonable definition of global weak solutions with a non-energy class of initial data, including the aforementioned Lorentz space. This paper can be regarded as an extension of a similar problem regarding the Lebesgue space $L_3$ to the weak Lebesgue space $L^{3,\infty}$, whose norms are both scale invariant with the respect to the Navier-Stokes scaling. 
\end{abstract}
\setcounter{equation}{0}

\section{Introduction}
In our paper we consider the Cauchy problem for the Navier-Stokes system in the space-time domain $Q_\infty=\mathbb R^3\times ]0,\infty[$ for vector-valued function $v=(v_1,v_2,v_3)=(v_i)$ and scalar function $q$, satisfying the equations
\begin{equation}\label{directsystem}
\partial_tv+v\cdot\nabla v-\Delta v=-\nabla q,\qquad\mbox{div}\,v=0
\end{equation}
in $Q_\infty$,
the boundary conditions
\begin{equation}\label{directbc}
v(x,t)\to 0\end{equation}
as $|x|\to\infty$ for all  $t\in  [0,\infty[$,
and the initial conditions
\begin{equation}\label{directic}
v(\cdot,0)=u_0(\cdot)
\end{equation}
with divergence free function $u_0$ belonging to a weak $L_3(\mathbb R^3)$ space denoted in the paper as  $L^{3,\infty}(\mathbb R^3)$.


 Let us recall the definition of the Lorentz spaces. 
For a measurable function $f:\Omega\rightarrow\mathbb R^m$ define:
\begin{equation}\label{defdist}
d_{f,\Omega}(\alpha):=|\{x\in \Omega : |f(x)|>\alpha\}|.
\end{equation}
Let $s\in ]0,\infty[$ and $l\in ]0,\infty]$. Given a measurable  $\Omega\subseteq\mathbb{R}^{n}$,  the Lorentz space $L^{s,l}(\Omega)$ is the set of all measurable functions $g$ on $\Omega$ such that the quasinorm $\|g\|_{L^{s,l}(\Omega)}$ is finite. Here:

\begin{equation}\label{Lorentznorm}
\|g\|_{L^{s,l}(\Omega)}:= \Big(s\int\limits_{0}^{\infty}\alpha^{l}d_{g,\Omega}(\alpha)^{\frac{l}{s}}\frac{d\alpha}{\alpha}\Big)^{\frac{1}{l}},
\end{equation}
\begin{equation}\label{Lorentznorminfty}
\|g\|_{L^{s,\infty}(\Omega)}:= \sup_{\alpha>0}\alpha d_{g,\Omega}(\alpha)^{\frac{1}{s}}.
\end{equation}

As is the case for the  $L_3(\mathbb R^3)$ norm, the Lorentz norm $L^{3,\infty}(\mathbb R^3)$ is scale invariant with respect to the Navier-Stokes scaling
$$v^\lambda(x,t)=\lambda v(\lambda x,\lambda^2t),\qquad  q^\lambda(x,t)=\lambda^2 q(\lambda x,\lambda^2t).$$
The important difference between the above spaces is that the norm in the space  $L_3(\mathbb R^3)$ possesses a shrinking property, i.e., the norm over a ball vanishes as the radius of this ball goes to zero, while the Lorentz space $L^{3,\infty}(\mathbb R^3)$ does not meet such a property for it's norm. The difference  can also be expressed in terms of the density of smooth compactly supported functions. 
The special interest of the space $L^{3,\infty}(\mathbb R^3)$ as a phase space for the Navier-Stokes equations is due to the fact that, in contrast to the space $L_3(\mathbb R^3)$, it contains minus one homogeneous divergence free functions.

The local in time  existence of strong solutions to the above Cauchy problem is a relatively well known fact proved in a number of papers, see for example,  \cite{cannone1997}, \cite{FujKato1964}, \cite{GigaMiy1989}, \cite{KoTa2001},    \cite{KozYam1994},  \cite{Plan1996}, and \cite{Taylor1992},  with the help of Kato's arguments \cite{Kato1984}. The typical outcome is  local in time existence of the so-called mild solutions under certain assumptions on the initial data\footnote{We are not discussing global existence of mild solutions for small initial data which is a very interesting topic itself but outside of our scope.}. This technique has a perturbative character as it does not take into account the skew symmetry of nonlinear term in full generality. Consequently, there  is an absence of results about global solvability for  large initial data\footnote{We restrict our considerations to three dimensional case only.}.

A breakthrough result in this direction has been established  by Lemarie-Rieusset, see \cite{LR1}. He showed that, for a very wide class of initial data\footnote{The completion of smooth compactly supported divergence free functions in the space $L_{s,{\rm unif}}$ with the finite norm $\|u\|_{s,{\rm unif}}:=\sup\limits_{x\in\mathbb R^3}\|u\|_{L_s(B(x,1))}$ for $s=2$.},
there exists a certain global solution to the initial value problem (\ref{directsystem})-(\ref{directic}) 
that in addition satisfies the local energy inequality. Such a solution exists globally in time if  $u_0\in L^{3,\infty}(\mathbb R^3)$. However, the class of Lemarie-Rieusset's solutions seems to be too wide and one can expect that additionally some global norms are bounded if 
more restrictive classes of initial data with unbounded energy\footnote{$u_0\notin L_2(\mathbb R^3)$.} are considered.

Moreover, there are some additional requirements for the class of weak global solutions. First of them is some kind of stability with respect to weak or weak-(*) convergence of initial data. To be precise, in our case, this would mean the following. Assuming that  a sequence $u_0^{(k)}$ converges weakly-(*) to $u_0$ in $L^{3,\infty}(\mathbb R^3)$, we need to show that the corresponding solutions $u^{(k)}$ with initial data $u_0^{(k)}$ converges in a sense to a solution $u$ with initial data $u_0$. This issue appears if one wants to show that scale invariant norms blow up as time approaches potential
blowup time, see \cite{Ser12} and \cite{BarkerSer}. The second important point is that the conception of Lemarie-Rieusset solutions has not been developed yet for unbounded domains different to the whole space $\mathbb R^3$. This makes it desirable to have a notion of weak global solutions that can be extended to other unbounded domains.

In the paper \cite{sersve2016}, the notion of global weak $L_3$-solutions has been introduced in the case of initial data belonging to the Lebesgue space $L_3(\mathbb R^3)$, which respects the  above two requirements. In addition, in \cite{sersve2016},
 regularity of weak $L_3$-solutions has been proven on a finite time interval,  the length of which depends on the initial data, that in turn implies uniqueness of weak $L_3$-solutions on this finite time interval. The aim of the paper is to implement this program in the case of initial data belonging to the Lorentz space $L^{3,\infty}(\mathbb R^3)$.  

To define our weak solution, we need to introduce additional notation:
$$S(t)u_0(x)=\int\limits_{\mathbb R^3}\Gamma(x-y,t)u_0(y)dy,
$$
where $\Gamma$ is a known heat kernel, $V(x,t):=S(t)u_0(x)$;

$L_s(\Omega)$ is a Lebesgue space in $\Omega\subseteq\mathbb R^3$ so that $L_s(\Omega)=L^{s,s}(\Omega)$ and abbreviations $L_s:=L_s(\mathbb R^3)$ and $L^{s,l}:=L^{s,l}(\mathbb R^3)$ are used;

$J$ and $\stackrel{\circ}J{^1_2}$ are the completion of the space
$$C^\infty_{0,0}(\mathbb R^3):=\{v\in C^\infty_0(\mathbb R^3):\,\,{\rm div}\,v=0\}$$
with respect to $L_2$-norm and the Dirichlet integral
$$\Big(\int\limits_{\mathbb R^3} |\nabla v|^2dx\Big)^\frac 12,$$
 correspondingly. Additionally, we define the space-time domains $Q_T:=\mathbb R^3\times ]0,T[$ and $Q_\infty:=\mathbb R^3\times ]0,\infty[$.
\begin{definition}\label{globalL3inf}

 We say that  $v$ 
 is a  weak $L^{3,\infty}$-solution to Navier-Stokes IBVP in $Q_T$ 
  if 
\begin{equation}\label{weaksolutionsplitting}
v=V+u,
\end{equation}
with $u\in L_{\infty}(0,T; J)\cap L_{2}(0,T;\stackrel{\circ} J{^1_2})$  and   
there exists $q\in L_{\frac 32, {\rm loc}}(Q_T)$ such that $u$ and $q$  satisfy the perturbed Navier Stokes system in the sense of distributions:
\begin{equation}\label{perturbdirectsystem}
\partial_t u+v\cdot\nabla v-\Delta u=-\nabla q,\qquad\mbox{div}\,u=0
\end{equation}
in $Q_T$
Additionally, it is requiried that for any $w\in L_{2}$:
\begin{equation}\label{vweakcontinuity}
t\rightarrow \int\limits_{\mathbb R^3} w(x)\cdot u(x,t)dx
\end{equation}
is a continuous function on  $[0,T].$
Moreover, $u$ satisfies the energy inequality:
\begin{equation}\label{venergyineq}
\|u(\cdot,t)\|_{L_{2}}^2+2\int\limits_{0}^t\int\limits_{\mathbb R^3} |\nabla u(x,t')|^2 dxdt'\leqslant$$$$\leq 2 \int_{0}^t\int_{\mathbb R^3}(V\otimes u+V\otimes V):\nabla udxdt'
\end{equation}
for all $t\in [0,T]$.

Finally, it is required that $v$ and $q$ satisfy the local energy inequality.
Namely, for a.a. $t\in ]0,T[$,
\begin{equation}\label{localenergyinequality}
\int\limits_{\mathbb R^3}\phi(x,t)|v(x,t)|^2dx+2\int\limits_{0}^t\int\limits_{\mathbb R^3}\int\phi |\nabla v|^2 dxdt^{'}\leqslant$$$$\leqslant
\int\limits_{0}^{t}\int\limits_{\mathbb R^3}[|v|^2(\partial_{t}\phi+\Delta\phi)+v\cdot\nabla\phi(|v|^2+2q)] dxdt^{'}
\end{equation}
for all non negative functions $\phi\in C_{0}^{\infty}(Q_T)$.

$v$ is called a global weak $L^{3,\infty}$-weak solution if it is a weak solution in $Q_T$ for any $T>0$.
\end{definition}
\begin{remark}\label{initialdata}
One can see that the right hand side in the energy inequality (\ref{venergyineq})  is finite and thus
the function $u$ satisfies the initial  condition in the strong $L_2$-sense, i.e.,
$u(\cdot,t)\to 0$ in $L_2$. 
\end{remark}

With regards to $V$, we can show that $\|V(\cdot,t)-u_0\|_{L_s,{\rm unif}}\to0$ as $t\to0$ for any $s<3$. In general, $V(\cdot,t)$ does not tends to $u_0$ in $L^{3,\infty}$ which can be easily seen for minus one homogeneous initial data, see \cite{cannone1997}.

The main result of the paper reads the following.
\begin{theorem}\label{weak stability}
Let 
$u_{0}^{(k)}\stackrel{*}{\rightharpoonup} u_0$ in $L^{3,\infty}$ and let $v^{(k)}$ be a sequence of a global weak $L^{3,\infty}$-solutions to the Cauchy problem for the Navier-Stokes system with initial data $u_0^{(k)}$. Then there exists a subsequence still denoted $v^{(k)}$ that converges to a global weak $L^{3,\infty}$-solution $v$ to the Cauchy problem for the Navier-Stokes system with initial data $u_0$,
in the sense of distributions.
\end{theorem}
\begin{cor}\label{existenseglobal} There exists at least one global weak $L^{3,\infty}$-solution to the Cauchy problem 
(\ref{directsystem})-(\ref{directic}).
\end{cor}

It is worth noticing that the smooth forward self-similar solution, the existence of which has been proved recently 
in \cite{jiasverak2014}, is a global weak $L^{3,\infty}$-solution.\\
Certain uniqueness and regularity statements, regarding weak $L^{3,\infty}$-solutions, provide further justification of the definition of weak $L^{3,\infty}$-solutions.
 We start with conditional uniqueness results.

\begin{theorem}\label{smoothnessanduniqueness} Let $v$ be a global weak $L^{3,\infty}$-solution to the Cauchy problem for the Navier-Stokes equations with the initial data $u_0\in L^{3,\infty}$. There is a universal constant $\varepsilon_0>0$ with the following property. 
If 
\begin{equation}\label{initialdata1}
\limsup\limits_{R\to0}
\|u_0\|_{L^{3,\infty}(B(x_0,R))}< \varepsilon_0
\end{equation}
for any $x_0\in \mathbb R^3$
and 
\begin{equation}\label{smallness}
 \|v(\cdot,t)-u_0(\cdot)\|_{L^{3,\infty}(\mathbb R^3)}<\varepsilon_0
\end{equation}
holds for all  $t\in]0,T[$, then $v$ is of class $C^\infty$ in $Q_T$.

Moreover, if $\tilde v$ is another global weak $L^{3,\infty}$-solution to the Cauchy problem for the Navier-Stokes equations with the the same initial data $u_0$, 
then $\tilde v=v$ in $Q_T$.
\end{theorem}
\begin{cor}\label{uniqueness}
Let $v$ and $\tilde v$ be two global weak $L^{3,\infty}$-solution to the Cauchy problem for the Navier-Stokes equations with the same initial data $u_0$. Suppose that $v\in C([0,T];L^{3,\infty})$. Then $\tilde v=v$ in $Q_T$.
\end{cor}

As to regularity, we can state the following.
\begin{theorem}\label{regularity} Suppose that $u_0\in L^{3,\infty}$.
There exists a universal constant $\varepsilon>0$ such that 
if 
\begin{equation}\label{smallnessasmp}
\langle   V\rangle _{Q_T}:=\sup\limits_{0<t<T}t^\frac 15\|V(\cdot,t)\|_{L_5}\leq \varepsilon,
\end{equation}
where $V(\cdot,t)=S(t)u_0(\cdot)$, then
there exists a $v$ that is a weak $L^{3,\infty}$-solution to the Cauchy problem for the Navier-Stokes system in $Q_T$ and satisfies the  property
\begin{equation}\label{lps}
\langle v\rangle_{Q_T}<2\langle  V\rangle_{Q_T}.
\end{equation}

Moreover, the following estimate is valid
\begin{equation}\label{bound1}
\|v-V\|_{L_\infty(0,T;L_{3})}<\langle V\rangle_{Q_{T}}
\end{equation}
 
\end{theorem}

It is easy to verify that a solution of Theorem \ref{regularity} is infinitely smooth in $Q_T$.

Although the main condition (\ref{smallnessasmp}) holds for a wide class initial data, it does not work for large minus one homogeneous initial data, see details in \cite{cannone1997}.
 
Finally, there will be shown that under Kozono-Yamazaki condition, see \cite{KozYam1995}, any global weak $L^{3,\infty}$-solution is unique and smooth on a short time interval.
  
 \begin{pro}\label{kozonoyamazakitypecondition}
 Let $u_{0}\in L^{3,\infty}$.
 There exists an $\varepsilon_{3}>0$ such that if
 \begin{equation}\label{initialdatacondition}
 \limsup_{\alpha\rightarrow\infty}\alpha d_{u_0,\mathbb{R}^3}(\alpha)^{\frac{1}{3}}<\varepsilon_3  
 \end{equation}
 then there exists a $T=T(u_{0})>0$ such that all global weak $L^{3,\infty}$ solutions, with initial data $u_0\in L^{3,\infty}$, coincide on $Q_{T}$.
 
 \end{pro}

\setcounter{equation}{0}
\section{Preliminaries}

Now we state  a fact about Lorentz spaces concerning a decomposition. The proof can be found in \cite{BarkerSer}. This will be formulated as a Lemma. Analogous statement  is Lemma II.I proven by Calderon in \cite{Calderon90}. 
\begin{lemma}\label{Decomp}
Take $1< t<r<s\leqslant\infty$, and suppose that $g\in L^{r,\infty}(\Omega)$. For any $N>0$, we let 
$g^N_-:= g\chi_{|g|\leqslant N}$ and $g^N_+:= g-g^N_-.$
Then 
\begin{equation}\label{bddpartg}
\|g^N_-\|_{L_{s}(\Omega)}^{s}\leqslant\frac{s}{s-r}N^{s-r}\|g\|_{L^{r,\infty}(\Omega)}^{r}-N^{s}d_{g}(N)
\end{equation}
if $s<\infty$
, and 
\begin{equation}\label{unbddpartg}
\|g^N_+\|_{L_{t}(\Omega)}^{t}\leqslant \frac{r}{r-t}N^{t-r}\|g\|_{L^{r,\infty}(\Omega)}^{r}.
\end{equation}
Moreover, for $\Omega=\mathbb R^3$
, if $g\in L^{r,l}$ with $1\leqslant l\leqslant\infty$ and $\rm{div}\,\,g=0$, 
 then
$g= \bar{g}^{N}+\tilde{g}^{N}$ where $\bar g^{N}\in [C_{0,0}^{\infty}(\mathbb R^3)]^{L_{s}(\mathbb R^3)}$ with
\begin{equation}\label{g1divfree}
\|\bar g^{N}\|_{L_{s}}^s\leqslant \frac{Cs}{s-r}N^{s-r}\|g\|_{L^{r,\infty}}^{r}
\end{equation}
and $\tilde g^{N}\in [C_{0,0}^{\infty}(\mathbb R^3)]^{L_{t}(\mathbb R^3)}$ with \begin{equation}\label{g2divfree}
\|\tilde g^{N}\|_{L_{t}}^t\leqslant \frac{Cr}{r-t}N^{t-r}\|g\|_{L^{r,\infty}}^{r}.
\end{equation}
\end{lemma}
\begin{remark}\label{remarktoLemma2.1}
Looking at the proof of the second part of of Lemma 2, we can easily see that
\begin{equation}\label{boundsforcomponents}
\|\bar g^{N}\|_{L^{r,\infty}}+\|\tilde g^{N}\|_{L^{r,\infty}}\leq c(r) \|g\|_{L^{r,\infty}}.\end{equation}\end{remark}


Let us recall the well known properties of $L^{s,1}$, for $1<s<\infty$, such as separability and density of smooth compactly supported functions. Also, recall that 
$$(L^{s,1})' =L^{s',\infty}, \qquad s'=\frac s{s-1}.$$
The identification is as follows, if $f\in L^{s',\infty}$ and $g\in L^{s,1}$:
$$T_f(g)=\int\limits_{\mathbb R^3} fgdx.$$

The following proposition concerns weak-star approximation of 
$L^{3,\infty}$ functions.\begin{pro}\label{weak*approx}
Let $u_{0}\in L^{3,\infty}$ be divergence free, in the sense of distributions. 
Then there exists a sequence $u^{(k)}_{0}\in C_{0,0}^{\infty}(\mathbb R^3)$ such that 
$$u^{(k)}_{0}\stackrel{*}{\rightharpoonup} u_0$$ in $L^{3,\infty}$.
\end{pro}
The proof is based on the estimates of solutions to the Neumann boundary problem in the terms of the Lorentz space $L^{\frac 32,1}$.

Now, consider the following Cauchy problem for the heat equation 
\begin{equation}\label{heat}
\partial_tu-\Delta u=0
\end{equation}
in $Q_\infty$,
\begin{equation}\label{inidata}
u(\cdot,0)=u_0(\cdot)\in L^{3,\infty}
\end{equation}
in $\mathbb R^3$.

Let us recall some known facts about solution operators of $S(t)$ for the corresponding 
semi-group. Indeed, $u(\cdot,t)=V(\cdot,t)=S(t)u_0(\cdot)$. 
\begin{pro}\label{semigroupweakL3}
We have 
\begin{equation}\label{L3inftybdd}
\|S(t)u_{0}\|_{L^{3,\infty}}\leqslant C\|u_{0}\|_{L^{3,\infty}}.
\end{equation}
Moreover for $3<r<\infty$, $m,k\in \mathbb{N}$:
\begin{equation}\label{semigroupsolonnikovest}
\|\partial^m_t\nabla^k S(t)u_0\|_{L_{r}}\leqslant \frac{C\|u_{0}\|_{L^{3,\infty}}}{t^{m+\frac k 2+{\frac 3 2}(\frac{1}{3}-\frac{1}{r})}}.
\end{equation}
Furthermore for $1\leqslant q<3$ the following limits exist as $t\rightarrow 0$:
\begin{equation}\label{converginitialdataLq}
\|S(t)u_{0}-u_{0}\|_{L_{q,unif}}\rightarrow 0,
\end{equation}
\begin{equation}\label{weak*converg}
S(t)u_{0}\stackrel{*}{\rightharpoonup} u_0 
\end{equation}
in $L^{3,\infty}$.
Under the additional constraint that $u_0\in \mathbb{L}^{3,\infty}:=[C^\infty_{0,0}]^{L^{3,\infty}}$ Then  we have that $S(t)u_{0}\in\mathbb{L}^{3,\infty}$ and 
\begin{equation}\label{strongconvergLL}
\lim_{t\rightarrow 0}\|S(t)u_{0}-u_{0}\|_{L^{3,\infty}}=0.
\end{equation}

\end{pro}
\begin{proof} The first two estimates are follows from convolution structure of the heat potential and the corresponding inequalities.

Recall the definition
$$\|f\|_{L_{p,unif}}:=\sup_{x_0\in\mathbb R^3}\|f\|_{L_{p}(B(x_0,1))}.$$
Now, let us focus only on proving (\ref{converginitialdataLq}), as all other statements follow from this and (\ref{L3inftybdd}).
From Lemma \ref{Decomp} we can write
\begin{equation}\label{decompu_01}
u_{0}:=\bar u_{0}^{1}+\tilde u_{0}^{1},
\end{equation} so that $$\bar u_{0}^{1}\in
 [C_{0,0}^{\infty}]^{L_{s}}\cap L^{3,\infty},\qquad \tilde u_{0}^{1}\in [C_{0,0}^{\infty}]^{L_{q}}\cap L^{3,\infty}$$with $1<q<3<s<\infty$.  It is clear that
$$\lim_{t\rightarrow 0}\|S(t)\bar u_{0}^{1}-\bar u_{0}^{1}\|_{L_{s}}=0,$$
$$\lim_{t\rightarrow 0}\|S(t)\tilde u_{0}^{1}-\tilde u_{0}^{1}\|_{L_{q}}=0.$$
From here, (\ref{converginitialdataLq}) is obtained without difficulty.
\end{proof}
\begin{pro}\label{semigroupweak*stabilitypro}
Let 
$$u_{0}^{(k)}\stackrel{*}{\rightharpoonup} u_0 $$
in $L^{3,\infty}$.
Then, for any $\phi\in C_{0}^{\infty}(Q_\infty)$:
\begin{equation}\label{semiweak*stabstatement}
\int\limits_0^\infty\int\limits_{\mathbb R^3} S(t)u_{0}^{(k)}(x)\phi(x,t) dx dt\rightarrow\int\limits_0^\infty\int\limits_{\mathbb R^3} S(t)u_{0}(x)\phi(x,t) dx dt.
\end{equation}
\end{pro}
\begin{proof}
By Lemma \ref{Decomp}, we have
$$
u_{0}^{(k)}:=\bar u_{0}^{(k)1}+\tilde u_{0}^{(k)1}
$$
and 
$$\sup_{k}\|\bar u_{0}^{(k)1}\|_{L_s}+\sup_{k}\|\tilde u_{0}^{(k)1}\|_{L_q}\leqslant C(s,q)\sup_{k}\|u_0^{(k)}\|_{L^{3,\infty}}.$$
It is clear that $\bar u_{0}^{(k )1}\rightharpoonup \bar u_0$, $S(t)\bar u_{0}^{(k)1} \rightharpoonup S(t)\bar u_0$ in $L_{s}$ and $\tilde u_{0}^{(k )1}\rightharpoonup \tilde u_0$, $S(t)\tilde u_{0}^{(k)1} \rightharpoonup S(t)\tilde u_0$ in $L_{q}$.
Obviously, $u_{0}= \bar u_0+\tilde u_0$. From here the conclusion is easily reached. 
\end{proof}

\setcounter{equation}{0}
\section{ Existence of global weak $L^{3,\infty}(\mathbb R^3)$-solutions }
\subsection{Apriori estimates }

Let $L_{s,l}(Q_T)$, $W^{1,0}_{s,l}(Q_T)$, 
$W^{2,1}_{s,l}(Q_T)$ be anisotropic 
(or parabolic) Lebesgues and Sobolev spaces 
with  norms
$$\|u\|_{L_{s,l}(Q_T)}=\Big(\int\limits_0^T\|u(\cdot,t)\|_{L_s}^ldt\Big)^\frac 1l,\quad \|u\|_{W^{1,0}_{s,l}(Q_T)}=\|u\|_{L_{s,l}(Q_T)}+\|\nabla u\|_{L_{s,l}(Q_T)},$$
$$\|u\|_{W^{2,1}_{s,l}(Q_T)}=\|u\|_{L_{s,l}(Q_T)}+\|\nabla u\|_{L_{s,l}(Q_T)}+\|\nabla^2 u\|_{L_{s,l}(Q_T)}+\|\partial_tu\|_{L_{s,l}(Q_T)}.$$
\begin{lemma}\label{integabilitynonlinearity}
Assume that $u\in L_{\infty}(0,T; J)\cap L_{2}(0,T;\stackrel{\circ}J{^1_2})$  and let $u_0\in L^{3,\infty}$ be divergence free. 
 Then
\begin{equation}\label{nonlinestsemigroup}
V\cdot\nabla V\in L_{\frac{11}{7}}(Q_T),
\end{equation}
\begin{equation}\label{nonlinest}
V\cdot\nabla u+u\cdot\nabla V\in L_{\frac{5}{4},\frac{3}{2}}(Q_T),
\end{equation}
\begin{equation}\label{ONeilest}
V\otimes u:\nabla u
\in L_1(Q_T).
\end{equation}
\end{lemma}
\begin{proof}
By Holder inequality and Proposition \ref{semigroupweakL3}:
$$\int\limits_{\mathbb R^3} |V\cdot\nabla V|^{\frac{11}{7}}dx\leqslant \|V\|_{L_{\frac{22}{7}}}^{\frac{11}{7}}\|\nabla V\|_{L_{\frac{22}{7}}}^{\frac{11}{7}}\leqslant$$$$\leqslant c\frac{\|u_{0}\|_{L^{3,\infty}}^{\frac{22}{7}}}{t^{\frac 6 7}}.$$
From here, (\ref{nonlinestsemigroup}) is easily established.
 Again, by Holder inequality and Proposition \ref{semigroupweakL3}:
 $$\|u\cdot\nabla V\|_{L_{\frac{5}{4}}}\leqslant \|\nabla V\|_{L_{\frac{10}{3}}}\|u\|_{L_{2}}\leqslant c\frac{\|u_0\|_{L^{3,\infty}}\|u\|_{L_{2,\infty}(Q_T)}}{t^{\frac{11}{20}}}.$$
 From this it is immediate that $u\cdot\nabla V\in L_{\frac{5}{4},\frac{3}{2}}(Q_T).$
 Again by Holder inequality, it is not difficult to verify
 $$\int\limits_0^T\|V\cdot\nabla u\|_{L_{\frac{5}{4}}}^{\frac{3}{2}}dt\leqslant (\int\limits_0^T \|\nabla u\|_{L_2}^2 dt)^{\frac{3}{4}} (\int\limits_0^T \|V\|_{L_{\frac{10}{3}}}^6 dt)^{\frac{1}{4}}.$$
 The conclusion is easily reached by noting that Proposition \ref{semigroupweakL3} gives:
$$\|V\|_{L_{\frac{10}{3}}}^6\leqslant c\frac{\|u_0\|_{L^{3,\infty}}}{t^{\frac{6}{20}}}.$$
The last estimate is known and shows why there are difficulties to prove energy estimate for $u$. 
By O'Neil's inequality and Proposition \ref{semigroupweakL3}:
$$\int\limits_{\mathbb R^3}| V\otimes u:\nabla u |dx\leqslant \|V\|_{L^{3,\infty}}\|  u\|_{L^{6,2}}\|\nabla u\|_{L_2}\leqslant$$$$
\leqslant c\|u_0\|_{L^{3,\infty}}\|\nabla u\|_{L_2}^2
.$$
We have used 
fact that
$L^{6,2}(\Omega)\hookrightarrow W^1_2(\Omega)$. See \cite{Adams}, for example.
\end{proof}

The next statement is a direct consequence of Lemma \ref{integabilitynonlinearity} and coercive estimates of solutions to the Stokes problem.
\begin{lemma}\label{higherderivativeandpressure}
Let $v$ be a global weak $L^{3,\infty}$-solution with functions $u$ and $q$ as in Definition \ref{globalL3inf}. Then
\begin{equation}\label{vdecomp}
(u,q)=\sum_{i=1}^3(u^i,p_{i})
\end{equation}
such that for any finite $T$: \begin{equation}\label{vdecompspaces}
(u^i,\nabla p_{i})\in W^{2,1}_{s_i,l_i}(Q_T)\times L_{s_i,l_i}(Q_T)
\end{equation} and
\begin{equation}\label{integindices}
(s_1,l_1)=( 9/8,3/2), s_2=l_2=11/7, (s_3,l_3)=(5/4,3/2).
\end{equation}
In addition $(u^i, p_{i})$ satisfy the following:
\begin{equation}\label{v_1eqn}
\partial_t u^1-\Delta u^1+\nabla p_{1}= -u\cdot\nabla u,
\end{equation}
\begin{equation}\label{v_2eqn}
\partial_t u^2-\Delta u^2+\nabla p_{2}= -V\cdot\nabla V,
\end{equation}
\begin{equation}\label{v_3eqn}
\partial_t u^3-\Delta u^3+\nabla p_{3}= -V\cdot\nabla u-u\cdot\nabla V
\end{equation}
in $Q_\infty$, and 
\begin{equation}\label{freediv}
{\rm div}\,u^i=0
\end{equation}
in $Q_\infty$ for $i=1,2,3$,
\begin{equation}\label{partinitial}
u^i(\cdot,0)=0
\end{equation}
for all $x\in \mathbb R^3$ and $i=1,2,3$.
\end{lemma}

Before the next Lemma let us introduce some notation. Let $u,\,v$ and $u_0$ be as in Definition \ref{globalL3inf}. Let $u_0= \bar u^{N}_0+\tilde u^{N}_0$ denote the splitting from Lemma \ref{Decomp}. Let us define the following: 
 \begin{equation}\label{V>N}
 \bar V^{N}(\cdot,t):=S(t)\bar u^{N}_0(\cdot,t),
 \end{equation}
 \begin{equation}\label{V<N}
 \tilde V^{N}(\cdot,t):=S(t)\tilde u^{N}_0(\cdot,t)
 \end{equation}
 and
 \begin{equation}\label{w>N}
 w^{N}(x,t):=u(x,t)+\tilde V^{N}(x,t).
 \end{equation}
\begin{lemma}\label{energyinequalitysplitting} 
 In the above notation, we have the following global energy inequality 
 \begin{equation}\label{w>Nenergyineq}
\|w^N(\cdot,t)\|_{L_2}^2+2\int\limits_0^t\int\limits_{\mathbb R^3} |\nabla w^N(x,t')|^2 dxdt'\leqslant$$$$\leqslant \|\tilde u_0^N\|_{L_2}^2+ 
2 \int_0^t\int\limits_{\mathbb R^3}(\bar V^N\otimes w^N+\bar V^N\otimes \bar V^N):\nabla w^N dxdt'
\end{equation}
that is valid for positive $N$ and $t$.
\end{lemma}
\begin{proof}
The first stage is showing that $w^N$ satisfies the local energy inequality. Let us briefly sketch how this can be done. Let $\phi\in C^\infty_0(Q_\infty)$ be a positive function. Observe that the assumptions in Definition \ref{globalL3inf} imply that the following function
\begin{equation}\label{crossterm}
t\rightarrow \int\limits_{\Omega}w^N(x,t)\cdot\bar V^N(x,t)\phi(x,t)dx
\end{equation}
is continuous for all $t\geq 0$.
It is not so difficult to show that this term has the following expression:
\begin{equation}\label{crosstermexpression1}
\int\limits_{\mathbb R^3}w^N(x,t)\cdot \bar V^N(x,t)\phi(x,t)dx= \int\limits_0^t\int\limits_{\mathbb R^3}(w^N\cdot \bar V^N)(\Delta\phi+\partial_t\phi) dxdt'-$$$$-
2\int\limits_0^t\int\limits_{\mathbb R^3}\nabla w^N:\nabla \bar V^N\phi dxdt'+\int\limits_0^t\int\limits_{\mathbb R^3} \bar V^N\cdot\nabla\phi q dxdt'+$$$$+
\frac 1 2\int\limits_0^t\int\limits_{\mathbb R^3} (|v|^2-|w^N|^2)v\cdot\nabla\phi dxdt'-$$$$-\int\limits_0^t\int\limits_{\mathbb R^3}(\bar V^N\otimes w^N+\bar V^N\otimes \bar V^N):\nabla w^N\phi dxdt'-$$$$-
\int\limits_0^t\int\limits_{\mathbb R^3}(\bar V^N\otimes \bar V^N+\bar V^N\otimes w^N):(w^N\otimes \nabla\phi)dxdt'.
\end{equation}
It is also readily shown that
\begin{equation}\label{crosstermexpression2}
\int\limits_{\mathbb R^3}|\bar V^N(x,t)|^2\phi(x,t)dx=\int\limits_0^t\int\limits_{\mathbb R^3}|\bar V^N(x,t')|^2(\Delta \phi(x,t')+\partial_t\phi(x,t'))dxdt'-$$$$-2\int\limits_0^t\int\limits_{\mathbb R^3} |\nabla\bar V^N|^2\phi dxdt'.
\end{equation}
Using (\ref{localenergyinequality}), together with (\ref{crosstermexpression1})-(\ref{crosstermexpression2}), we obtain that for all $t\in ]0,\infty[$ and for all non negative functions $\phi\in C_{0}^{\infty}(Q_\infty)$:
\begin{equation}\label{localenergyinequalityW>N}
\int\limits_{\mathbb R^3}\phi(x,t)|w^N(x,t)|^2dx+2\int\limits_{0}^t\int\limits_{\mathbb R^3}\phi |\nabla w^N|^2 dxdt^{'}\leqslant$$$$\leqslant
\int\limits_{0}^{t}\int\limits_{\mathbb R^3}|w^N|^2(\partial_{t}\phi+\Delta\phi)+2qw^N\cdot\nabla\phi dxdt^{'}+$$$$+
\int\limits_0^t\int\limits_{\mathbb R^3}|w^N|^2v\cdot\nabla\phi dxdt'+$$$$+
2\int\limits_0^t\int\limits_{\mathbb R^3} (\bar V^N\otimes \bar V^N+\bar V^N\otimes w^N):(\nabla w^N\phi+w^N\otimes\nabla\phi)dxdt'
\end{equation}
In the next part of the proof, let $\phi(x,t)=\phi_{1}(t)\phi_{R}(x)$. Here, $\phi_1\in C_{0}^{\infty}(0,\infty)$ and $\phi_{R}\in C_0^{\infty}(B(2R))$ are positive functions. Moreover, $\phi_{R}=1$ on $B(R)$, $0\leqslant\phi_{R}\leqslant 1$,
$$|\nabla \phi_R|\leqslant c/R,$$
$$|\nabla^2\phi_R|\leqslant c/R^2.$$
Since $\tilde u_0^N\in [C_{0,0}^{\infty}(\mathbb R^3)]^{L_{2}(\mathbb R^3)}$, it is obvious that for $\tilde V^N(\cdot,t):=S(t)\tilde u_0^N(\cdot,t)$ we the energy equality:
\begin{equation}\label{V>Nenergyequality}
\|\tilde V^N(\cdot,t)\|_{L_2}^2+\int\limits_0^t\int\limits_{\mathbb R^3}|\nabla \tilde V^N|^2 dxdt'=\|\tilde u_0^N\|_{L_2}^2.
\end{equation}
By semigroup estimates, we have  for $2\leqslant p\leqslant \infty$, $10/3\leqslant q\leqslant \infty$:
\begin{equation}\label{V>Nsemigroupest}
\|\tilde V^N(\cdot,t)\|_{L_p}\leqslant \frac{C(p)}{t^{\frac{3}{2}(\frac{1}{2}-\frac{1}{p})}}\|\tilde u_0^N\|_{L_2},
\end{equation}
\begin{equation}\label{V<Nsemigroupest}
\|\bar V^N(\cdot,t)\|_{L_q}\leqslant \frac{C(q)}{t^{\frac{3}{2}(\frac{3}{10}-\frac{1}{q})}}\|\bar u_0^N\|_{L_{\frac{10}{3}}}.
\end{equation}
Hence, we have $w^N\in C_{w}([0,T]; J)\cap L_{2}(0,T;\stackrel{\circ}{J}{^1_2})$.
Here, $T$ is finite and $C_{w}([0,T]; J)$ denotes continuity with respect to the weak topology.
Using these facts, and usual multiplicative inequalities, it is obvious that the following limits hold:
$$\lim_{R\rightarrow\infty}\int\limits_{\mathbb R^3}\phi_R(x)\phi_1(t)|w^N(x,t)|^2dx+2\int\limits_{0}^t\int\limits_{\mathbb R^3}\phi_R\phi_1 |\nabla w^N|^2 dxdt^{'}=$$$$=
\int\limits_{\mathbb R^3}\phi_1(t)|w^N(x,t)|^2dx+2\int\limits_{0}^t\int\limits_{\mathbb R^3}\phi_1 |\nabla w^N|^2 dxdt^{'},$$
$$\lim_{R\rightarrow\infty}\int\limits_{0}^t\int\limits_{\mathbb R^3}(|w^N|^2\partial_t\phi_{1}\phi_R+2(\bar V^N\otimes w^N+\bar V^N\otimes\bar V^N):\nabla w^N\phi_{1}\phi_{R})dxdt'=$$$$=
\int\limits_{0}^t\int\limits_{\mathbb R^3}(|w^N|^2\partial_t\phi_{1}+2(\bar V^N\otimes w^N+\bar V^N\otimes\bar V^N):\nabla w^N\phi_{1})dxdt',$$
$$\lim_{R\rightarrow\infty}\int\limits_{0}^t\int\limits_{\mathbb R^3}(|w^N|^2\phi_{1}\Delta\phi_R+\phi_1|w^N|^2 v\cdot\nabla \phi_{R}+$$$$+2\phi_1(\bar V^N\otimes w^N+\bar V^N\otimes\bar V^N):(w^N\otimes\nabla\phi_R))dxdt'=0. $$
Let us focus on the term containing the pressure, namely
$$ \int\limits_0^t\int\limits_{\mathbb R^3}qw^N\cdot\nabla\phi_R\phi_1 dxdt^{'}.$$
Define $T(R):= B(2R)\setminus B(R)$. We can instead treat
$$ \int\limits_0^t\int\limits_{T_{+}(R)}(q-[q]_{B(2R)})w^N\cdot\nabla\phi_R\phi_1 dxdt^{'}.$$
Using Poincare inequality, it is not so difficult to show:

\begin{equation}\label{pv1termLEI}
|\int\limits_0^t\int\limits_{T(R)}(p_1-[p_1]_{B(2R)})w^N\cdot\nabla\phi_R\phi_1 dxdt^{'}|\leqslant$$$$\leqslant \frac{C\|\phi_1\|_{L_{\infty}(0,t)}}{R^{\frac{2}{3}}}\|w^N\|_{L_{3}(T(R)\times ]0,t[)}\|\nabla p_1\|_{L_{\frac{9}{8},\frac{3}{2}}(Q_t)}
,
\end{equation}
\begin{equation}\label{pv2termLEI}
|\int\limits_0^t\int\limits_{T(R)}(p_2-[p_2]_{B(2R)})w^N\cdot\nabla\phi_R\phi_1 dxdt^{'}|\leqslant$$$$\leqslant C\|\phi_1\|_{L_{\infty}(0,t)}\|w^N\|_{L_{\frac{11}{4}}(T(R)\times ]0,t[)}\|\nabla p_2\|_{L_{\frac{11}{7}}(Q_t)}
,
\end{equation}
\begin{equation}\label{pv3termLEI}
|\int\limits_0^t\int\limits_{T_{+}(R)}(p_3-[p_3]_{B(2R)})(w^N\cdot\nabla\phi_R)\phi_1 dxdt^{'}|\leqslant$$$$\leqslant \frac{C\|\phi_1\|_{L_{\infty}(0,t)}}{R^{\frac{2}{5}}}\|w^N\|_{L_{3}(T(R)\times ]0,t[)}\|\nabla p_3\|_{L_{\frac{5}{4},\frac{3}{2}}(Q_t)}
.
\end{equation}
Using (\ref{pv1termLEI})-(\ref{pv3termLEI}), multiplicative inequalities and properties of the pressure decomposition in Definition \ref{globalL3inf} we infer that
$$\lim_{R\rightarrow \infty}\int\limits_0^t\int\limits_{T(R)}qw^N\cdot\nabla\phi_R\phi_1 dxdt^{'}=0.$$
Thus, putting everything together, we get for arbitrary positive function $\phi_1\in C_{0}^{\infty}(0,\infty)$:
\begin{equation}\label{w>Nenergyineqcompacttime}
\int\limits_{\mathbb R^3}\phi_1(t)|w^N(x,t)|^2dx+2\int\limits_{0}^t\int\limits_{\mathbb R^3}\phi_{1}(t) |\nabla w^N|^2 dxdt^{'}\leqslant$$$$\leqslant
\int\limits_{0}^{t}\int\limits_{\mathbb R^3}|w^N|^2\partial_{t}\phi_{1}+
2(\bar V^N\otimes w^N+\bar V^N\otimes \bar V^N):\nabla w^N\phi_1 dxdt'
\end{equation}
From Remark  \ref{initialdata}, we see that
\begin{equation}\label{w>Nconverginitialdata}
\lim_{t\rightarrow0}\|w^N(\cdot,t)-\tilde u_0^N(\cdot)\|_{L_{2}}=0.
\end{equation}
Using known arguments from \cite{BarkerSer}, we have the following estimates:
\begin{equation}\label{Barkserest1}
\int\limits_{0}^t\int\limits_{\mathbb R^3}|\bar V^N\otimes w^N:\nabla w^N|dxdt'\leqslant$$$$\leqslant
 CN^{\frac{1}{10}}\|u_0\|_{L^{3,\infty}}^{\frac{9}{10}}\left(\int\limits_{0}^t\int\limits_{\mathbb R^3} |\nabla w^N|^2dxdt'\right)^{\frac{4}{5}}\left(\int\limits_0^
t\frac{\|w^N(\cdot,\tau)\|^{2}_{L_{2}}}{\tau^{\frac{3}{4}}}d\tau\right)^{\frac{1}{5}}
,
\end{equation}
\begin{equation}\label{Barkserest2}
\int\limits_{0}^t\int\limits_{\mathbb R^3}|\bar V^N\otimes\bar V^N:\nabla w^N| dxdt'\leqslant Ct^{\frac{7}{20}}N^{\frac{1}{5}}\|u_0\|_{L^{3,\infty}}^{\frac{9}{5}}\|\nabla w^N\|_{L_2(Q_t)}.
\end{equation}
Using (\ref{w>Nenergyineqcompacttime}), (\ref{w>Nconverginitialdata}) and (\ref{Barkserest1})-(\ref{Barkserest2}), we infer (\ref{w>Nenergyineq}) by standard arguments involving an appropriate choices of $\phi_1(t)=\phi_{\epsilon}(t)$ and letting $\epsilon$ tend to zero.

\end{proof}
\begin{lemma}\label{venergyest}
Let $u$, $v$ and $u_0$ be as in Definition \ref{globalL3inf}. Then the following estimate is valid for all $N,t>0$:
\begin{equation}\label{venergybdd}
\|u(\cdot,t)\|_{L_{2}}^2+\int\limits_0^t\int\limits_{\mathbb R^3} |\nabla u|^2dxdt'\leqslant C(N^{-1}\|u_0\|_{L^{3,\infty}}^3+t^{\frac{7}{10}}N^{\frac{2}{5}}\|u_0\|_{L^{3,\infty}}^{\frac{18}{5}})+$$$$+
C\exp(Ct^{\frac{1}{4}}N^{\frac{1}{2}}\|u_0\|_{L^{3,\infty}}^{\frac{9}{2}})(N^{-\frac{1}{2}}t^{\frac{1}{4}}\|u_0\|_{L^{3,\infty}}^{\frac{33}{8}}+t^{\frac{19}{20}}N^{\frac{9}{10}}\|u_0\|_{L^{3,\infty}}^{\frac{199}{40}}).
\end{equation}
Hence, taking $N= t^{-\frac{1}{2}}$ gives the following scale invariant estimate:
\begin{equation}\label{venergybddscaled}
\|u(\cdot,t)\|_{L_{2}}^2+\int\limits_0^t\int\limits_{\mathbb{R}^3} |\nabla u|^2dxdt'\leqslant$$$$\leqslant
C t^{\frac{1}{2}}\exp(C \|u_0\|_{L^{3,\infty}}^{\frac{9}{2}})(\|u_0\|_{L^{3,\infty}}^{\frac{9}{8}}+ 1)(\|u_0\|_{L^{3,\infty}}^{3}+\|u_0\|_{L^{3,\infty}}^{\frac{18}{5}}).
\end{equation}
\end{lemma}
\begin{proof}
First observe that  $u=w^N- \tilde V^N$. Thus, using (\ref{V>Nenergyequality}) we see that
$$\|u(\cdot,t)\|_{L_{2}}^2+\int\limits_0^t\int\limits_{\mathbb R^3} |\nabla u|^2dxdt'\leqslant$$$$\leqslant 2\|\tilde u^N_0\|_{L_{2}}^2+2\|w^N(\cdot,t)\|_{L_{2}}^2+2\int\limits_0^t\int\limits_{\mathbb R^3} |\nabla w^N|^2dxdt'. $$
By (\ref{g2divfree}):
\begin{equation}\label{u0>NL2norm}
\|\tilde u_0^N\|_{L_{2}}^2\leqslant CN^{-1}\|u_{0}\|_{L^{3,\infty}}^3.
\end{equation}
Thus, it is sufficient to prove (\ref{venergybdd}) for $w^N$ in place of $u$. From now on, denote 
$$y_N(t):=\|w^N(\cdot,t)\|_{L_{2}}^2.$$
Using (\ref{w>Nenergyineq}), estimates (\ref{Barkserest1})-(\ref{Barkserest2}), (\ref{u0>NL2norm}) and the Young's inequality obtain that
$$y_N(t)+\int\limits_0^t\int\limits_{\mathbb R^3} |\nabla w^N|^2dxdt'\leqslant CN^{\frac{1}{2}}\|u_{0}\|_{L^{3,\infty}}^{\frac{9}{2}}\int\limits_0^t \frac{y_N(\tau)}{\tau^{\frac{3}{4}}}d\tau+$$$$+C(N^{-1}\|u_0\|_{L^{3,\infty}}^3+t^{\frac{7}{10}}N^{\frac{2}{5}}\|u_0\|_{L^{3,\infty}}^{\frac{18}{5}}).
$$
The conclusion is then easily reached using a Gronwall type Lemma.
\end{proof}

\subsection{Existence of global weak $L^{3,\infty}(\mathbb R^3)$-solutions}

{\bf Proof of Theorem \ref{weak stability}} We have
$$u_{0}^{(k)}\stackrel{*}{\rightharpoonup} u_0$$ in $L^{3,\infty}$ and
may assume that 
$$M:=\sup\limits_k\|u_0^{(k)}\|_{L^{3,\infty}}<\infty.$$

Firstly, define
$$V^{(k)}(\cdot,t):= S(t)u_0^{(k)}(\cdot,t),\qquad V(\cdot,t):=S(t)u_{0}(\cdot,t).$$
By Proposition \ref{semigroupweak*stabilitypro}, we see that $V^{(k)}$ converges to $V$ on $Q_\infty$ in the sense of distributions.
By Proposition \ref{semigroupweakL3}, we see that
\begin{equation}\label{VkweakL3est}
\|V^{(k)}(\cdot,t)\|_{L^{3,\infty}}\leqslant CM,
\end{equation}
\begin{equation}\label{Vksemigroupest}
\|\partial^m_t\nabla^l V^{(k)}(\cdot,t)\|_{L_{r}}\leqslant \frac{CM}{t^{m+\frac l 2+{\frac 3 2}(\frac{1}{3}-\frac{1}{r}})}.
\end{equation}
Here $r\in ]3,\infty]$.
 For $T<\infty$ and $l\in ]1,\infty[$, we have the compact embedding
$$W^{2,1}_{l}(B(n)\times ]0,T[)\hookrightarrow C([0,T]; L_{l}(B(n))).$$
From this and (\ref{Vksemigroupest}) one immediately infers that for every $n\in\mathbb{N}$ and $l\in ]1,\infty[$:
\begin{equation}\label{V^kstrongconverg}
\partial^m_t\nabla^l V^{(k)}\rightarrow\partial^m_t\nabla^l V\,\,\,{\rm in}\,\,\, C([{1}/{n},n];L_{l}(B(n))).
\end{equation}

Fixing $N=1$ in Lemma \ref{venergyest} we have:
\begin{equation}\label{v^kenergybdd}
\|u^{(k)}(\cdot,t)\|_{L_{2}}^2+\int\limits_0^t\int\limits_{\mathbb R^3} |\nabla u^{(k)}|^2dxdt'\leqslant f_{0}(M,t).
\end{equation}
By means of a Cantor diagonalisation argument, we can abstract a subsequence such that for any finite $T>0$:
\begin{equation}\label{v_kweak*}
u^{(k)}\stackrel{*}{\rightharpoonup}u\,\,\, {\rm in}\,\,\,L_{2,\infty}(Q_T),
\end{equation}
\begin{equation}\label{gradv_kweak}
\nabla u^{(k)}{\rightharpoonup}\nabla u\,\,\, {\rm in}\,\,\,L_{2}(Q_T).
\end{equation}
Using (\ref{v_kweak*}), together with (\ref{venergybddscaled}), we also get that:
\begin{equation}\label{vzeronearinitialtime}
\|u\|_{L_{2,\infty}(Q_t)}\leqslant C(M)t^{\frac{1}{2}}.
\end{equation}

From (\ref{v^kenergybdd}) it is easily inferred that
\begin{equation}\label{v^k9/83/2bdd}
\|u^{(k)}\cdot\nabla u^{(k)}\|_{L_{\frac{9}{8},\frac{3}{2}(Q_t)}}\leqslant
f_{1}(M,t).
\end{equation}
By the same reasoning as in Lemma \ref{integabilitynonlinearity}, we obtain:
\begin{equation}\label{seqnonlin11/7}
\|V^{(k)}\cdot \nabla V^{(k)}\|_{L_{\frac{11}{7}}(Q_t)}\leqslant f_{2}(M,t),
\end{equation}
\begin{equation}\label{seqnonlin5/43/2}
 \|V^{(k)}\cdot\nabla u^{(k)}+u^{(k)}\cdot\nabla V^{(k)}\|_{L_{\frac{5}{4},\frac{3}{2}}(Q_t)}\leqslant f_{3}(M,t).
 \end{equation}
 Split $u^{(k)}=\sum_{i=1}^3 u^{i(k)}$ according to Definition \ref{globalL3inf}, namely (\ref{vdecomp}).
 By coercive estimates for the Stokes system, along with (\ref{v^k9/83/2bdd}) obtain:
 \begin{equation}\label{v^k_1est}
\|u^{1(k)}\|_{W^{2,1}_{\frac 9 8,\frac 3 2}(Q_t)}+\| \nabla p^{(k)}_1\|_{L_{\frac 9 8,\frac 3 2}(Q_t)}\leqslant Cf_1(M,t),
 \end{equation}
 \begin{equation}\label{v^k_2est}
 \|u^{2(k)}\|_{W^{2,1}_{\frac{11}{7}}(Q_t)}+\|\nabla p^{(k)}_2\|_{L_{\frac{11}{7}}(Q_t)}\leqslant Cf_2(M,t),
 \end{equation}
 \begin{equation}\label{v^k_3est}
 \|u^{3(k)}\|_{W^{2,1}_{\frac 5 4,\frac 3 2}(Q_t)}+\|\nabla p^{(k)}_3\|_{L_{\frac 5 4,\frac 3 2}(Q_t)}
 \leqslant Cf_3(M,t).
 \end{equation}
 By the previously mentioned embeddings, we infer from (\ref{v^k_1est})-(\ref{v^k_3est}) that for any $n\in\mathbb{N}$ we have the following convergence for a certain subsequence:
 \begin{equation}\label{v^kstrongconverg1}
 u^{(k)}\rightarrow u\,\,\,{\rm in}\,\,\, C([0,n]; L_{\frac{9}{8}}(B(n)).
 \end{equation}
 Hence, using (\ref{v^kenergybdd}), it is standard to infer that for any $s\in ]1,10/3[$
 \begin{equation}\label{v^kstrongconverg2}
 u^{(k)}\rightarrow u\,\,\,{\rm in}\,\,\, L_{s}(B(n)\times ]0,n[).
 \end{equation}
 It is also not so difficult to show that for any $f\in L_{2}$ and for any $n\in\mathbb N$:
 \begin{equation}\label{v^kweakcontconverg}
 \int\limits_{\mathbb R^3} u^{(k)}(x,t)\cdot f(x)dx\rightarrow \int\limits_{\mathbb R^3} u(x,t).f(x)dx\,\,\,{\rm in}\,\,\,C([0,n]).
 \end{equation}
 Using (\ref{vzeronearinitialtime}) with (\ref{v^kweakcontconverg}), we establish that
 \begin{equation}\label{vstrongestinitialtime}
 \lim_{t\rightarrow 0}\|u(\cdot,t)\|_{L_{2}}=0.
 \end{equation}
 All that remains to show is establishing the local energy inequality (\ref{localenergyinequality}) for the limit and establishing the energy inequality (\ref{energyinequalitysplitting}) for $u$.
 Verifying  the local energy inequality is not so difficult and hence omitted.
 Let us focus on verifying (\ref{energyinequalitysplitting}) for $u$. 
By identical reasoning to Lemma \ref{energyinequalitysplitting}, we have that for an arbitrary positive function $\phi_1(t)\in C_0^{\infty}(0,\infty)$:
\begin{equation}\label{venergyineqcompacttime}
\int\limits_{\mathbb R^3}\phi_1(t)|u(x,t)|^2dx+2\int\limits_{0}^t\int\limits_{\mathbb R^3}\phi_{1}(t) |\nabla u|^2 dxdt^{'}\leqslant$$$$\leqslant
\int\limits_{0}^{t}\int\limits_{\mathbb R^3}|u|^2\partial_{t}\phi_{1}+
2 (V\otimes  u+V\otimes V):\nabla u\phi_1 dxdt'.
\end{equation}
From Lemma \ref{integabilitynonlinearity} and semigroup estimates, we have that 
$$(V\otimes u+V\otimes V):\nabla u\in L_1(Q_T)$$
for any  positive finite $T$.
Using these facts and (\ref{vstrongestinitialtime}), the conclusion is reached by choosing appropriate $\phi_{\epsilon}=\phi_1$ and taking a limit. $\Box$

Let us comment on Corollary \ref{existenseglobal}. Recall that by Proposition \ref{weak*approx}, there exists 
a sequence  $u^{(k)}_{0}\in 
 C_{0,0}^{\infty}(\mathbb R^3)
 $
 such that 
$$u_{0}^{(k)}\stackrel{*}{\rightharpoonup} u_0$$ in $L^{3,\infty}$. 
It was shown in \cite{sersve2016} that for any $k$
there exists a global $L_{3}$-weak solution $v^{(k)}$. Now, Corollary \ref{existenseglobal} follows from Theorem \ref{weak stability}.

\setcounter{equation}{0}
\section{Uniqueness }
First we introduce the notation $Q(z_0,R)= B(x_0,R)\times ]t-R^2,t[.$
Here, $z_0=(x_0,t)\in Q_{\infty}.$

{\bf Proof of Theorem \ref{smoothnessanduniqueness}}
Step I. Regularity. Our first remark is that, given $\varepsilon>0$ and $R>0$, 
there exists a number
$R_*(T,R,\varepsilon)  >0$ such that
if $B(x_0,R)\subset \mathbb R^3\setminus B(R_*)$ and $t_0-R^2>0$ then
$$\frac 1{R^2}\int\limits_{Q(z_0,R)}(|v|^3+|q-[q]_{B(x_0,R)}|^\frac 32)dx dt \leq \varepsilon .$$
For $v$ 
 it is certainly true. For $q$, we can use Lemma \ref{energyinequalitysplitting} 
 Indeed, if $q=p_{1}+p_{2}+p_{3}$, then, for example, we have
$$\frac 1{R^2}\int\limits_{Q(z_0,R)}|p_{1}-[p_{1}]_{B(x_0,R)}|^\frac 32dxds\leq
$$
$$ \leq\frac 1{R^2}\int\limits^T_0\int\limits_{B(x_0,R)}|p_{1}-[p_{1}]_{B(x_0,R)}|^\frac 32dxds\leq \frac 1{R^\frac 32}\int\limits^{T}_{0}\Big(\int\limits_{B(x_0,R))}|\nabla p_{1}|^\frac 98dx\Big)^\frac 43dt\leq $$$$\leq \frac 1{R^\frac 32}\int\limits^{T}_{0}\Big(\int\limits_{\mathbb R^3\setminus B(R_*))}|\nabla p_{1}|^\frac 98dx\Big)^\frac 43dt\to0$$
as $R_*\to\infty$ for any fixed $R>0$. Since 
the pair $v$ and $q$ satisfies the local energy inequality,
by $\varepsilon$-regularity theory developed    in \cite{CKN}, we can claim that
$$|v(z_0)|\leq \frac cR$$
as long as $z_0$ and $R$ satisfy the conditions above. 

Now, our aim is to show that $v$ is locally bounded. To this end, we can use condition (\ref{initialdata1})
and state that there exists $R_0(x_0,\varepsilon_0)>0$ such that
$$\|u_0\|_{L^{3,\infty}(B(x_0,R))}<\varepsilon_0$$
for all $0<R<R_0(x_0,\varepsilon_0)$.
Then

$$\|v(\cdot,t)\|_{L^{3,\infty}(B(x_0,R))}\leq 
\|u_0\|_{L^{3,\infty}(B(x_0,R))}+\varepsilon_0<
2\varepsilon_0$$
for all $0<R<R(x_0,\varepsilon_0)$ and for all $t\in ]0,T[$.
.   

 Using H\"older inequality for Lorentz spaces, we have 
$$\frac 1r\Big(\int^{t_0}_{t_0-r^2}\Big(\int\limits   
_{B(x_0,r)}|v|^2dx\Big)^2dt\Big)^\frac 14\leq
$$ $$\leq c\sup\limits_{t_0-r^2<t<t_0}\|v(\cdot,t)\|_{L^{3,\infty}(B(x_0,r))}\leq c\varepsilon_0$$
for all $t_0\in ]0,T]$, for all $0<r<R_0(x_0,\varepsilon_0)$ satisfying $t_0-r^2>0$, and $c$ is a positive universal constant. Then the local boundedness follows from $\varepsilon$-regularity conditions derived in \cite{SerZa} with a suitable choice of the constant $\varepsilon_0$.

So, we can ensure that $v\in L_\infty(Q_{\delta,T})$ for any $\delta>0$. Here, $Q_{\delta,T}=\mathbb R^3\times ]\delta,T[$. Then, we can easily deduce that, for any $\delta>0$, $u\in W^{2,1}_2(Q_{\delta,T})$, $\nabla u\in L_{2,\infty}(Q_{\delta,T})$, and $\nabla q\in L_2(Q_{\delta,T})$.
By iterative arguments, we complete the proof of the theorem.

Step II. Uniqueness. Regularity results proved above
 allow us to state that   the energy identity
$$\frac 12\int\limits_{\mathbb R^3}|u(x,t)|^2dx+\int\limits^t_0\int\limits_{\mathbb R^3}|\nabla u|^2dxds= \int\limits^t_0\int\limits_{\mathbb R^3}V\otimes v:\nabla u
dxds$$
holds for any $t>0$ and, moreover,
$$ \int\limits_{\mathbb R^3}\Big(\partial_tu(x,t)\cdot w(x)+(v(x,t)\cdot\nabla v(x,t))\cdot w(x)+\nabla u(x,t):\nabla w(x)\Big)dx=0 $$
for any $w\in C^\infty_{0,0}(\mathbb R^3)$ and for all $t\in ]0,T[$.

Letting  $\tilde u=\tilde v-V$ and $w=\tilde u-u$, 
we can repeat the same arguments as in \cite{sersve2016} to obtain

$$\frac 12  \int\limits_{\mathbb R^3}|w(x,t_0)|^2dx+\int\limits_0^{t_0}\int\limits_{\mathbb R^3}|\nabla w|^2dxdt\leq $$
$$\leq \int\limits_0^{t_0}\int\limits_{\mathbb R^3}\Big(\tilde v\otimes \tilde v:\nabla w- v\otimes v:\nabla w\Big)dxdt=\int\limits_0^{t_0}\int\limits_{\mathbb R^3}(w\otimes v+v\otimes w):\nabla wdxdt.$$
So, finally,
$$I:=\int\limits_{\mathbb R^3}|w(x,t_0)|^2dx+\int\limits_0^{t_0}\int\limits_{\mathbb R^3}|\nabla w|^2dxdt\leq$$$$\leq c\int\limits_0^{t_0}\int\limits_{\mathbb R^3}|v|^2|w|^2dxdt.
$$
Let us fix $s\in ]0,T[$, then
$$I\leq cI_1+cI_2+cI_3,$$
where
$$I_1=\int\limits_0^{t_0}\int\limits_{\mathbb R^3}|v(x,t)-u_0(x)|^2|w(x,t)|^2dxdt, $$$$ I_2=\int\limits_0^{t_0}\int\limits_{\mathbb R^3}|v(x,s)-u_0(x)|^2|w(x,t))|^2dxdt, $$
$$I_3=\int\limits_0^{t_0}\int\limits_{\mathbb R^3}|v(x,s)|^2|w(x,t)|^2dxdt. $$ 

The first two integrals are evaluated in the same way with the help of the H\"older inequality for Lorentz spaces:
$$c(I_1+I_2)\leq c\int\limits^{t_0}_0(\|v(\cdot,t)-u_0(\cdot)\|^2_{L^{3,\infty}}+$$$$+\|v(\cdot,s)-u_0(\cdot)\|^2_{L^{3,\infty}})\|w(\cdot,t)\|^2_{L^{6,2}}dt.$$
By assumptions of the theorem, 
$$c(I_1+I_2)\leq c\varepsilon_0\int\limits^{t_0}_0\|w(\cdot,t)\|^2_{L^{6,2}}dt.$$
It remains to apply the Sobolev inequality and conclude
that
$$c(I_1+I_2)\leq c\varepsilon_0\int\limits^{t_0}_0\|\nabla w(\cdot,t)\|^2_{L_2}.$$

To estimate  $I_3$, we are going to use the fact that $v(\cdot,s)$ is bounded for positive $s\leq T$, i.e.,
$$\|v(\cdot,s)\|_{L_\infty}\leq g(s).$$
Here, it might happen that $g(s)\to\infty$ if $s\to0$. So,
$$I_3\leq g^2(s)\int\limits^{t_0}_0\int\limits_{\mathbb R^3}|w(x,t)|^2dxdt.$$
Then reducing  $\varepsilon_0$ if necessary, we find 
 
$$\int\limits_{\mathbb R^3}|w(x,t_0)|^2dx+\int\limits_0^{t_0}\int\limits_{\mathbb R^3}|\nabla w|^2dxdt\leq cg^2(s)\int\limits_0^{t_0}\int\limits_{\mathbb R^3}|w(x,t)|^2dxdt$$
for all $t_0\in ]0,T[$, which implies that $w(\cdot,t)=0$ for the same $t$. $\Box$

To justify Corollary \ref{uniqueness}, we can argue as follows. First, it can be shown that
$$\|u_0\|_{L^{3,\infty}(B(x_0,R))}\to0$$
as $R\to0$. Indeed, if $v$ is a weak $L^{3,\infty}$-solution in $Q_T$, then for a.a. $t\in ]0,T[$ we have $v(\cdot,t)\in L^{3,\infty}$ along with the following property. Namely, for all $x_0\in\mathbb R^3$:
$$\|v(\cdot,t)\|_{L^{3,\infty}(B(x_0,R))}\to0$$
as $R\to0$. Since it is assumed that $v\in C([0,T];L^{3,\infty})$, the above property in fact holds for all $t\in [0,T]$.  

Now, one should split the interval $[0,T]$
into sufficiently small pieces by points $t_k=kT/N$ with $k=1,2,...,N$ so that
$$\|v(\cdot,t)-v(\cdot,t_{k-1})\|_{L^{3,\infty}(\mathbb R^3)}<\varepsilon_0$$
for any $t\in [t_{k-1},t_k]$ and for all $k=1,2,...,N$. It remains to apply Theorem \ref{smoothnessanduniqueness} successively for $k=1,2,...,N$.

\setcounter{equation}{0}
\section{ Regularity}
{\bf Proof of Theorem \ref{regularity}}
We use the Kato iteration scheme. Let us  define the following,
 for $k=1,2,...,$
$$v^{(1)}=V,\qquad V^{(k+1)}=V+u^{(k+1)},
$$
where $u^{(k+1)}$ solves the following problem
$$\partial_tu^{(k+1)}-\Delta u^{(k+1)}+\nabla q^{(k+1)}=-{\rm div}\,v^{(k)}\otimes v^{(k)},\,\,\,\rm{div}\, u^{k+1}=0$$ in 
$Q_T$, 
$$u^{(k+1)}(\cdot,0)=0$$ 
in $\mathbb R^3$. It is easy to check that 
for solutions to the above linear problem the following estimates are
true 
$$\langle u^{(k+1)}\rangle_{Q_T}\leq c\langle v^{(k)}\rangle_{Q_T}^2,$$ 
$$\|u^{(k+1)}\|_{L_\infty(0,T;L_3)}\leq c\langle v^{(k)}\rangle_{Q_T}^2$$
and thus we have 
$$
\langle v^{(k+1)}\rangle_{Q_T}\leq \langle V\rangle_{Q_T}+c\langle v^{(k)}\rangle_{Q_T}^2,
$$
$$\|v^{(k+1)}\|_{L_\infty(0,T;L^{3,\infty})}\leq \|V\|_{L_\infty(0,T;L^{3,\infty})}+c\langle v^{(k)}\rangle _{Q_T}^2,$$
and 
$$\|v^{(k+1)}-V\|_{L_\infty(0,T;L_3)}\leq c\langle v^{(k)}\rangle_{Q_T}^2 
$$
for all $k=1,2,...$. Using Kato's arguments, one easily show that for  $\varepsilon<\frac{1}{4c}$ 
we shall have 
\begin{equation}\label{1stest}
\langle v^{(k)}\rangle_{Q_T}<2\langle V\rangle _{Q_T}  \end{equation}
for all $k=1,2,...$.
 We get, in addition, that
\begin{equation}\label{1stest1}
\|v^{(k)}\|_{L_\infty(0,T;L^{3,\infty})}\leq \|V\|_{L_\infty(0,T;L^{3,\infty})}+\langle V\rangle_{Q_T},\end{equation}
\begin{equation}\label{1stest2}
\|v^{(k+1)}-V\|_{L_\infty(0,T;L_3)}\leq  \langle V\rangle_{Q_T}\end{equation}
for all $k=1,2,...$.
Furthermore, Kato's arguments also give that there is a $v=V+u$ such that
\begin{equation}\label{converg1}
\langle v^{(k)}-v\rangle_{Q_T},\, \langle u^{(k)}-u\rangle_{Q_T}\rightarrow 0,
\end{equation}
\begin{equation}\label{converg2}
\|v^{(k)}-v\|_{L_{\infty}(0,T; L^{3,\infty})},\,\|u^{(k)}-u\|_{L_{\infty}(0,T; L_{3})}\rightarrow 0.
\end{equation} 
Next we note that by interpolation:
 \begin{equation}\label{interpolationineq}
 t^{\frac 1 8}\|g(\cdot,t)\|_{L_{4}}\leq  C(\|g(\cdot,t)\|_{L^{3,\infty}})^{\frac 3 8}(t^{\frac 1 5}\|g(\cdot,t)\|_{L_{5}})^{\frac{5}{8}}.
 \end{equation}
 Using this and (\ref{converg1})-(\ref{converg2}), we immediately see that
 \begin{equation}\label{converg3}
 \|v^{(k)}-v\|_{L_{4}(Q_T)},\,\|u^{(k)}-u\|_{L_{4}(Q_T)}\rightarrow 0.
 \end{equation}

We also can exploit our equation, together with the pressure equation, to derive  the following estimate for the energy and pressure:
\begin{equation}\label{energypresestimate}
\|u^{(k)}-u^{(m)}\|^2_{2,\infty,Q_T}+\|\nabla u^{(k)}-u^{(m)}\|^2_{2,Q_T}+\|q^{(k)}-q^{(m)}\|^2_{2,Q_T}\leq$$$$\leq c\int\limits^T_0\int\limits_{\mathbb R^3}|v^{(k)}\otimes v^{(k)}-v^{(m)}\otimes v^{(m)}|^2dxdt.
\end{equation}

Using (\ref{converg3}), we immediately see the following
 \begin{equation}\label{u^nstrongconverg}
 u^{(k)}\rightarrow u\,\,{\rm in}\,\,W^{1,0}_{2}(Q_T)\cap C([0,T];L_2(\mathbb{R}^3))\cap L_{4}(Q_T) 
 ,
 \end{equation}
 \begin{equation}\label{limnearinitialtime}
 u(\cdot,0)=0
 ,
 \end{equation}
 \begin{equation}\label{presconvergence}
 q^{(k)}\rightarrow q\,\,{\rm in}\,\, L_{2}(Q_T).
 \end{equation}
 Clearly, the pair $v$ and $q$ satisfies the Navier-Stokes equations, in a distributional sense.
 It is easily verified that 
 \begin{equation}\label{semigrouppropertiesgradient}
 S(t)u_0 \in L_{4}(Q_{T})\cap L_{2,\infty}(B(R)\times ]0,T[)\cap W^{1,0}_{2}(B(R)\times ]\epsilon,T[)
 \end{equation}
 for any $0<R$, $0<\varepsilon<T$. 
 By (\ref{u^nstrongconverg})-(\ref{presconvergence}), $v$ has the same property. It is known that this, along with $q\in L_{2}(Q_{T})$, is sufficient to infer that the pair $v$ and  $q$ satisfies the local energy equality. This can be shown by a mollification argument.
 Showing that $u$ satisfies the energy inequality (on $Q_{T}$) present in our definition of global weak $L^{3,\infty}$ solution  (in fact, in this case it is an equality), can now be carried out in a similar way to Lemma \ref{energyinequalitysplitting}. Here, certain decay properties of $u,q$ from (\ref{u^nstrongconverg})-(\ref{presconvergence}) are needed, as well as the fact that $\lim_{t\rightarrow 0^+}\|u(\cdot,t)\|_{L_{2}(\mathbb{R}^3)}=0$. $\Box$

 {\bf Proof of Theorem \ref{kozonoyamazakitypecondition}}
 Condition (\ref{initialdatacondition}) ensures that there exists an $N>0$ such that $$\| (u_0)^{N}_+\|_{L^{3,\infty}}<\varepsilon_3.$$
 Thus, by the convolution inequality,
 $$<S(t) (u_0)_+^{N}>_{Q_T},\,\|S(t) (u_0)_+^{N} \|_{L_{\infty}(0,T; L^{3,\infty})}<C\varepsilon_3.$$
 By Lemma \ref{Decomp},
  we have that
 $$ \|(u_0)_-^{N}\|_{L_5}\leq CN^{\frac 2 5}\|u_{0}\|_{L^{3,\infty}}^{\frac 3 5}.$$
 Thus,
 $$\langle V\rangle_{Q_T}<C\varepsilon_3+
T^{\frac{1}{5}}C N^{\frac{2}{5}}\|u_{0}\|_{L^{3,\infty}}^{\frac{3}{5}}.$$
 Taking $T:=T(u_0)$ and $\varepsilon_3$ sufficiently small gives, by Theorem \ref{regularity}, the existence of \ weak $L^{3,\infty}$ solution on $Q_{T}$ such that
 $$
 \|v-S_1(t)(u_0)_{-}^{N}\|_{L_{\infty}(0,T; L^{3,\infty})}\leq$$
 \begin{equation}\label{vclosebdd}\leq   \|v-V\|_{L_{\infty}(0,T; L^{3,\infty})}+\|S_1(t)(u_0)_{+}^{N}\|_{L_{\infty}(0,T; L^{3,\infty})}<$$$$
 <V>_{Q_T}+C\varepsilon_3<\varepsilon_0.  
 \end{equation}
 Next we  notice that $S_1(t)(u_0)_{-}^{N}$ is bounded in $(Q_{T})$ and moreover
 \begin{equation}\label{semigroupshirnking}
 \|S_1(t)(u_0)_{-}^{N}\|_{L^{3,\infty}(B(x_0,R))}\leq  C R N.
 \end{equation}
 These facts, along with (\ref{vclosebdd}), are enough to conclude by using minor adaptations to the proof of Theorem \ref{smoothnessanduniqueness}.
 $\Box$
 \begin{remark}\label{Tlowerbound1}
 Furthermore, there is the lower bound for T:
 \begin{equation}\label{Tlowerbound}
 T\geq \frac{\min(\varepsilon^5,\varepsilon_0^5)}{C N^2\|u_0\|^3_{L^{3,\infty}}}.
 \end{equation} 
 Here, $C$ is a universal constant. Moreover, $\varepsilon$ and $\varepsilon_0$ are from Theorems \ref{smoothnessanduniqueness} and  \ref{regularity} respectively. 
 \end{remark}

\end{document}